\documentclass{lms}
\pdfoutput=1

\usepackage[all,cmtip,line]{xy}

\usepackage{tikz}
\usepackage{enumerate}
\usepackage{microtype}

\def\arXiv#1{arXiv:\href{http://arXiv.org/abs/#1}{#1}}

\usepackage{amsmath, amstext, amssymb, amscd}
\usepackage{microtype}

\usepackage[colorlinks,bookmarks=false]{hyperref}
\hypersetup{citecolor=blue,linkcolor=blue,urlcolor=blue}

\newtheorem{theorem}{Theorem}

\newtheorem{proposition}[theorem]{Proposition}

\newnumbered{conjecture}{Conjecture}
\newnumbered{definition}{Definition}
\newnumbered{remark}[theorem]{Remark}
\newnumbered{example}{Example}

\newcommand{\sfcaption}[1]{\caption{{\sf #1}}}

\newcommand{\hh}{ {\mathbb H} }
\newcommand{\C}{{\mathbb C}}
\newcommand{\Z}{{\mathbb Z}}
\newcommand{\Q}{{\mathbb Q}}

\newcommand{\sO}{{\mathcal O}}
\newcommand{\sL}{{\mathcal L}}

\newcommand{\M}{\mathcal M}

\DeclareMathOperator{\Jac}{Jac}
\DeclareMathOperator{\Pic}{Pic}

\DeclareMathOperator{\Sym}{Sym}
\DeclareMathOperator{\tr}{tr}

\title[Real multiplication through explicit correspondences]
{Real multiplication through explicit correspondences}

\author{Abhinav Kumar and Ronen E. Mukamel}

\classno{14-04, 14H40 (primary), 32G15, 14G35 (secondary)}

\begin{document}
\maketitle

\begin{abstract}
We compute equations for real multiplication on the divisor classes of
genus two curves via algebraic correspondences.  We do so by
implementing van Wamelen's method for computing equations for
endomorphisms of Jacobians on examples drawn from the algebraic models
for Hilbert modular surfaces computed by Elkies and Kumar.  We also
compute a correspondence over the universal family for the Hilbert
modular surface of discriminant $5$ and use our equations to prove a
conjecture of A. Wright on dynamics over the moduli space of Riemann
surfaces.
\end{abstract}

\section{Introduction}
\label{sec:introduction}
Abelian varieties, their endomorphisms and their moduli spaces play a
central role in modern algebraic geometry and number theory. Their
study has important applications in a broad array of fields including
cryptography, dynamics, geometry, and mathematical physics.  Of
particular importance are the abelian varieties with extra
endomorphisms (other than those in $\Z$).  In dimension one, elliptic
curves with complex multiplication have been studied extensively.  In
this paper, we focus on curves of genus two whose Jacobians have real
multiplication by a real quadratic ring $\sO$.

For such a curve $C$ over a number field $K$, we use the ideas from
van Wamelen's work \cite{vW1,vW2,vW3} to explicitly compute the action
of real multiplication by $\sO$ on the divisors of $C$. In particular,
we determine equations for an algebraic correspondence on $C$, i.e., a
curve $Z$ with two maps $\phi, \psi : Z \to C$ such that the induced
endomorphism $T = \psi_* \circ \phi^*$ of $\Jac(C)$ generates
$\sO$. The discovery of the correspondence $Z$ uses floating-point
calculations on the analytic Jacobian $\Jac(C) \otimes \C$. We then
rigorously certify the real multiplication of $\sO$ on $\Jac(C)$ by
computing the action of $T$ on one-forms using exact arithmetic in
$K$, or a small degree extension of $K$. Combined with standard
equations for the group law on $\Jac(C)$, our techniques immediately
lead to an algebraic description the action of $\sO$ on degree zero
divisors of $C$.

This paper completes a research program initiated in \cite{EK} and
\cite{KM}.  Let $\M_{g,n}$ denote the moduli space of smooth genus $g$
curves with $n$ marked points, and, for each totally real order $\sO$,
denote by $\M_{g,n}(\sO)$ the locus of curves whose Jacobians admit
real multiplication by $\sO$. The paper \cite{EK} describes a method
for parametrizing the Humbert surface $\M_2(\sO) = \M_{2,0}(\sO)$ for
real quadratic $\sO$ as well as its double cover, the Hilbert modular
surface $Y(\sO)$. It also carries out the computation for $\sO =
\sO_K$, the ring of integers of every real quadratic field $K$ of
discriminant less than $100$, producing equations for the
corresponding Hilbert modular surfaces. The paper \cite{KM} describes
a method for computing the action of $\sO$ on the one-forms of curves
in $\M_2(\sO)$, and uses it in particular to compute algebraic models
for Teichm\"uller curves $\M_2$. Using these techniques one can
furnish equations defining curves $C \in \M_2(\sO)$, rigorously prove
that $\Jac(C)$ admits real multiplication by $\sO$ and rigorously
compute the action of $\sO$ on the one-forms of $C$.  In this paper we
solve the problem of describing the action of $\sO$ on $\Jac(C)$ as
algebraic morphisms by computing the action on the divisors of $C$.

\paragraph{Example for discriminant $5$.} 
To demonstrate our method, consider the genus two curve
\begin{equation}
\label{eqn:ex1}
 C : u^2 = t^5 - t^4 + t^3 + t^2 - 2t + 1.
\end{equation}
Equation \ref{eqn:ex1} was obtained from the equations in \cite{EK}.
Using the methods of this paper, we can formulate and prove the
following theorem.
\begin{theorem} \label{thm:ex1}
Let $\phi: Z \to C$ be the degree two cover of $C$ of Equation \ref{eqn:ex1} defined by
\[ t^2 x^2-x-t+1 = 0. \]
The curve $Z$ is of genus $6$ and admits an additional map $\psi : Z \to
C$ of degree two.  The induced endomorphism $T = \psi_* \circ \phi^*$ of
$\Jac(C)$ is self-adjoint with respect to the Rosati involution,
satisfies $T^2-T-1=0$ and generates real multiplication by $\sO_5$.
\end{theorem}
\noindent In Section \ref{sec:examples} we give several other examples
of correspondences on particular genus two curves of varying
complexity.  In Section \ref{sec:disc5family}, we describe how to
implement our method in families and compute a correspondence for the
entire Hilbert modular surface $Y(\sO_5)$.

\paragraph{Divisor classes supported at eigenform zeros.} 
In the universal\footnote{Technically, it would be more accurate to
  use the term ``universal'' to describe the family over the stack
  whose underlying coarse moduli space is the Hilbert modular
  surface. However, we will indulge in this mild abuse of notation
  throughout this paper.} Jacobian over the space $\M_{g,n}(\sO)$,
there is a natural class of multisection obtained from $\sO$-linear
combinations of divisor classes supported at eigenform zeros and
marked points.  Filip recently showed that such divisors play a
pivotal role in the behavior of geodesics in moduli space \cite{F}.
As an application of our equations for real multiplication, we prove a
theorem about such divisors over $\M_{2,1}(\sO_5)$ and prove a
conjecture of Wright on dynamics over the moduli space of curves.

In the universal Jacobian over $\M_{2,1}(\sO_5)$, let $L$ be the
multisection whose values at the curve $C$ marked at $P
\in C$ are divisors of the form
\begin{equation}
\label{eqn:multisection}
(P-Z_1) - T \cdot (Z_2 - Z_1) \in \Jac(C)
\end{equation}
where $T$ is a Rosati invariant endomorphism of $\Jac(C)$ satisfying $T^2-T-1 =0$ and $Z_1$ and $Z_2$
are the zeros of a $T$-eigenform $\omega$ on $C$.  The various choices of endomorphism $T$, eigenform $\omega$ and ordering of the zeros of $\omega$ give four points in $\Jac(C)$ for generic $(C,P) \in \M_{2,1}(\sO_5)$.  The vanishing of
$L$ defines a closed subvariety of $\M_{2,1}(\sO_5)$:
\begin{equation}
\M_{2,1}(\sO_5;L) = \left\{ (C,P) \in \M_{2,1}(\sO_5) : \parbox[c]{8em}{some branch of $L$ vanishes at $(C,P)$ } \right\}.
\end{equation}
\noindent One might expect $\M_{2,1}(\sO_5;L)$ to be a curve in the
threefold $\M_{2,1}(\sO_5)$ since the relative dimension of the
universal Jacobian is two.  We use our equations for real
multiplication to show that $\M_{2,1}(\sO_5;L)$ is unexpectedly large.
\begin{theorem}
\label{thm:eformzeros}
The space $\M_{2,1}(\sO_5;L)$ is an irreducible surface in $\M_{2,1}$.
\end{theorem}
To relate Theorem \ref{thm:eformzeros} to dynamics, recall that
$\M_{g,n}$ carries a Teichm\"uller metric and that every vector
tangent to $\M_{g,n}$ generates a complex geodesic, i.e., a
holomorphic immersion $\hh \to \M_{g,n}$ which is a local isometry.
McMullen proved that the locus $\M_2(\sO)$ is the closure of a complex
geodesic in the moduli space $\M_2$ of unmarked genus two curves for
each real quadratic $\sO$ \cite{Mc1}.  A corollary of Theorem
\ref{thm:eformzeros} shows that $\M_{2,1}(\sO_5;L)$ enjoys the same
property.
\begin{theorem}
\label{thm:cxgeod}
There is a complex geodesic $f : \hh \to \M_{2,1}$ with
$\overline{f(\hh)} = \M_{2,1}(\sO_5;L)$.
\end{theorem}
\noindent In other words, there is a dynamically natural way to choose
finitely many points on each curve in $\M_2(\sO_5)$.  This was
originally conjectured by Wright, and will be proven by other means in the
forthcoming paper \cite{EMMW}.

By Filip's characterization of the behavior of complex geodesics in
moduli space \cite{F}, every complex geodesic is dense in a subvariety
of $\M_{g,n}(\sO)$ characterized by $\sO$-linear relations among
divisor classes supported at eigenform zeros and marked points.  Theorem
\ref{thm:cxgeod} is the first example of such a subvariety where a
relation involving a ring strictly larger than $\Z$ appears.  

\paragraph{Prior work on equations for real multiplication.}  
Several authors \cite{Wi1,Wi2,Sa1,HS1,HS2,Sa2} have given geometric
descriptions of real multiplication based on Humbert's work on
Poncelet configurations of conics \cite{Hu,vdG,Ja}. Our work combines
explicit examples or families of genus $2$ curves (obtained from the
equations for Hilbert modular surfaces computed in \cite{EK}) with the
method of van Wamelen outlined in \cite[Section 3]{vW2}.  The latter
uses high-precision numerical sampling on the Jacobian and subsequent
linear algebra to find explicit algebraic descriptions of isogenies
for Jacobians of genus $2$ curves. We describe the isogenies through
correspondences; see \cite{Sm} for other examples of equations for
real multiplication obtained through correspondences.  Furthermore, we
address the challenge of computing correspondences in families,
thereby giving a description of real multiplication for the universal
abelian surface over a Hilbert modular surface.

\paragraph{Outline.} 
In Section \ref{sec:background}, we recall some basic facts about
Jacobians of curves, their endomorphisms and correspondences.  In
Section \ref{sec:discovery}, we describe our method for finding the
equations of a correspondence associated to a Jacobian endomorphism.
In Section \ref{sec:oneforms}, we describe how to compute the induced
action on one-forms and thereby certify the equations obtained by the
method in Section \ref{sec:discovery}.  In Section \ref{sec:examples},
we give several examples of varying complexity of explicit
correspondences.  In Section
\ref{sec:disc5family}, we address challenges to implementing our
algorithm in families, and describe a correspondence for the entire
Hilbert modular surface for discriminant $5$.  In Section
\ref{sec:eigenformzeros}, we discuss the applications to dynamics and
prove Theorems \ref{thm:eformzeros} and \ref{thm:cxgeod}.

\paragraph{Computer files.}  
Auxiliary files containing computer code to verify the calculations in
this paper are available from
\url{http://arxiv.org/abs/1602.01924}. To access these, download the
source file for the paper. It is a {\sf tar} archive, which can be
extracted to produce not only the \LaTeX~ file for this paper, but
also the computer code. The text file {\sf README.txt} gives a brief
guide to the various auxiliary files.

\paragraph{Acknowledgments.}  
We thank Curt McMullen for many helpful conversations and suggestions,
Noam Elkies and Alex Wright for useful comments on an earlier draft of
this paper, and the anonymous referees for a careful reading of the
paper and numerous helpful comments. REM was supported in part by
National Science Foundation grant DMS-1103654.

\section{Background} \label{sec:background}
In this section, we recall some general facts about curves, their
Jacobians and algebraic correspondences. We will work over the complex
numbers. The basic reference for this section is \cite{BL}.

\paragraph{Jacobians.}  
Let $C$ be a smooth projective curve of genus $g$ over $\C$.  The
holomorphic one-forms on $C$ form a $g$-dimensional vector space
$\Omega(X)$.  Integration gives rise to an embedding $H_1(C,\Z) \to
\Omega(C)^*$ and the image of this embedding is a lattice.  The
quotient
\[ \Jac(C) = \Omega(C)^* / H_1(C,\Z) \]
is a compact, complex torus called the Jacobian of $C$. The symplectic
intersection form on $H_1(C,\Z)$ induces a principal polarization on
the torus $\Jac(C)$ (i.e., an isomorphism of this abelian variety with
its dual).

\paragraph{The Abel-Jacobi map.}  
Let $\Pic^0(C)$ denote the group of degree zero divisors on $C$ up to
linear equivalence. Integration gives rise to an isomorphism of
groups
\[ AJ: \Pic^0(C) \to \Jac(C). \]
This is the {\em Abel-Jacobi} map. When $\Pic^0(C)$ is thought of as
the complex points of the Picard variety of $C$, this map is an
isomorphism of abelian varieties over $\C$.

\paragraph{The theta divisor.}  
Choosing a base point $P_0 \in C$ allows us to define a birational map
$\xi$ from the $g$th symmetric power of $C$ to $\Pic^0(C)$ via
the formula
\[ \xi\left(  \left\{ P_1,\dots,P_g \right\} \right)  = \Big(\sum_i P_i\Big) - gP_0. \]
The divisor $ \left\{ S \in \Sym^g(C) : P_0 \in S \right\}$ gives rise
to a divisor $\Theta$ on $\Jac(C)$ called the {\em theta divisor}.

\paragraph{Pullback and pushforward.} 
Now consider a holomorphic map $\psi : Z \to C$ between curves.  The map
$\psi$ induces a map $\Omega(C) \to \Omega(Z)$ whose dual covers
a holomorphic homomorphism
\[ \psi_* : \Jac(Z) \to \Jac(C). \]
Under the identification of Jacobians with degree zero divisors via
the Abel-Jacobi map, $\psi_*$ corresponds to the pushforward of
divisors, i.e.,
\begin{equation}
\label{eqn:pushforward}
 \psi_*\Big( \sum_i P_i - \sum_i Q_i \Big) = \sum_i \psi(P_i) - \sum_i \psi(Q_i). 
\end{equation}
We call $\psi_*$ the {\em pushforward map}.  The map $\psi$ also induces a
{\em pullback map}
\[ \psi^* : \Jac(C) \to \Jac(Z) \]
obtained as the dual map to $\psi_*$ by identifying the Jacobians of
$C$ and $Z$ with their corresponding duals via their principal
polarizations. For non-constant $\psi$, we can obtain $\psi^*$ at the
level of divisors by summing along fibers, i.e.,
\begin{equation}
\label{eqn:pullback}
 \psi^* \Big( \sum_i P_i - \sum_i Q_i \Big) = \sum_i \psi^{-1} \left( P_i \right) -\sum_i \psi^{-1}(Q_i),
\end{equation}
while for a constant map $\psi^* = 0$. The composition $\psi_* \circ
\psi^*$ is the multiplication by $\deg(\psi)$ map on $\Jac(C)$.

\paragraph{Correspondences.}  
A {\em correspondence} $Z$ on $C$ is a holomorphic curve in $C \times
C$.  Fix a correspondence $Z$ on $C$ and let $\phi = \pi_1$ and $\psi
= \pi_2$ be the the two projection maps from $Z$ to $C$. The
correspondence $Z$ gives rise to an endomorphism of $\Jac(C)$ via the
formula $T = \psi_* \circ \phi^*$.  From Equations
\ref{eqn:pushforward} and \ref{eqn:pullback}, we see that $T$ acts on
divisors of the form $P-Q$ by the formula
\begin{equation}
\label{eqn:Tondiv}
T(P-Q) = \psi(\phi^{-1}(P)) - \psi(\phi^{-1}(Q)).
\end{equation}
Such divisors generate $\Pic^0(C)$, so Equation \ref{eqn:Tondiv}
determines the action of $T$ on $\Pic^0(C)$.

Conversely, every endomorphism $T$ of the Jacobian endomorphism arises
via a correspondence. To see this, we embed $C$ in $\Jac(C)$ via the
map $P \mapsto P-P_0$ (note that when the genus of $C$ is two, the
resulting cycle is just the theta divisor $\Theta$).  Since the image
of $C$ under this embedding generates the group $\Jac(C)$, the
restriction of $T$ to $C$ determines $T$. This map $T|_C$ is a
$C$-valued point of $\Jac(C) \cong \Pic^0_C$, and by the functorial
property of the Picard variety, it corresponds to a line bundle $\sL$
on $C \times C$, whose fibers $\sL|_{C \times P}$ are all of degree
$0$. Then we can take $Z$ to be an effective divisor corresponding to
the line bundle $\sL \otimes \pi_1^*(\sO_C(gP_0))$. Concretely, the
intersection of $Z$ with $P \times C$ consists of points $(P,Q_i)$
with $Q_1,\dots,Q_g \in C$ satisfying $T\cdot (P-P_0) = \sum_i Q_i
-gP_0$. Using this fact, it is easily checked that the two
constructions are inverse to each other.

\section{Computing equations for correspondences} \label{sec:discovery}
In this section, we describe our method for discovering
correspondences.  The methods in this section are numerical and rely
on floating point approximation.  Nonetheless, the correspondences we
obtain are presented by equations with exact coefficients lying in a
number field.  In Section \ref{sec:oneforms}, we will describe how to
certify these equations using only rigorous integer arithmetic in
number fields to prove theorems about real multiplication on genus two
Jacobians.

\paragraph{Setup.}
Our starting point is a fixed curve $C$ of genus two {\em known} to
have a Jacobian endomorphism $T$ generating real multiplication by the
real quadratic order $\sO_D$ of discriminant $D$.  Such curves can be
supplied by the methods in \cite{EK}.  We assume that $C$ and $T$ are
defined over a number field $K$ and that $C$ is presented as a
hyperelliptic curve
\begin{equation}
C : u^2 = h(t) \mbox{ with } h \in K[t] \mbox{ monic},\, \deg(h) = 5.
\end{equation}
We fix an embedding $K \subset \C$ so that we can base change to $\C$
and work with the analytic curve $C^{an}$ and the analytic Jacobian
$J^{an} = \Jac(C^{an})$. For simplicity we have assumed in this
section that $h$ is monic of degree $5$ so that $C$ has a $K$-rational
Weierstrass point $P_0$ at infinity.  We discuss below how to handle
the sextic case (see Remark \ref{rmk:sextic}).

\paragraph{Analytic Jacobians in Magma.} 
The computer system {\sf Magma} has several useful functions for
working with analytic Jacobians and their endomorphisms, implemented
by van Wamelen. An excellent introduction may be found in \cite{vW2},
and extensive documentation is available in the {\sf Magma} handbook
\cite{Mag}. The relevant functions for us are:
\begin{enumerate}[(a)]
\item {\sf AnalyticJacobian} (see also {\sf BigPeriodMatrix}):
  computes the periods of $dt/u$ and $t\,dt/u$ in $\Omega(C)$,
  yielding a numerical approximation to the period matrix
  $\Pi(C^{an})$ and a model for the analytic Jacobian $J^{an} =
  \C^2/\Pi(C^{an}) \cdot \Z^4$.
\item {\sf EndomorphismRing}: computes generators for the endomorphism
  ring of $J^{an}$.  Each endomorphism $T^{an}$ is presented as a pair
  of matrices $T^{an}_\Omega \in M_2(\C)$ and $T^{an}_\Z \in M_4(\Z)$
  satisfying $\Pi(C^{an}) \cdot T_\Z = T_\Omega \cdot \Pi(C^{an})$ (up
  to floating point precision).
\item {\sf ToAnalyticJacobian}: computes the Abel-Jacobi map by
  numerical integration.
\item {\sf FromAnalyticJacobian}: computes the inverse of the
  Abel-Jacobi map using theta functions.
\end{enumerate}

\paragraph{Discovering correspondences.}
We compute equations defining the correspondence $Z$ on
$C$ associated to $T$ is as follows.
\begin{enumerate}[(1)]
\item Compute the analytic Jacobian $J^{an}$ and an endomorphism
  $T^{an}$ generating real multiplication.
\item Choose low height points $P_i = (t_i,u_i) \in C^{an}$ with $t_i
  \in K$.
\item For each $i$, numerically compute points $R_i = (t(R_i),u(R_i))$
  and $Q_i=(t(Q_i),u(Q_i))$ in $C^{an}$ such that
\[ AJ(Q_i+R_i-2P_0) = T^{an} \cdot (AJ(P-P_0)). \]
\item For each $i$, compute the exact coefficients of the polynomial
  $F_i(x) = (x-t(Q_i))(x-t(R_i))$ in $K[x]$ using LLL.
\item Interpolate to determine a polynomial $F \in K(t)[x]$ which
  specializes to $F_i$ at $t=t(P_i)$ and let $Z$ be the degree two
  cover defined by $F$, i.e., with
\begin{equation}
K(Z) = K(C)[x]/(F).
\end{equation}
\end{enumerate}
To realize $Z$ as a divisor in $C \times C$, we need to compute a
square root for $h(x)$ in $K(Z)$.  For small examples, this can be
done by working in the function field of $Z$.  In general, we revisit
steps $(3)$ and $(4)$ and do the following.
\begin{enumerate}[(1)]
\setcounter{enumi}{5}
\item For each $i$, determine $u(Q_i)$ as a $K$-linear combination of
  $u(P_i)$ and $u(P_i)t(Q_i)$,
\item Interpolate to determine a rational function $y \in K(Z)$ which
  is a $K(t)$-linear combination of $u$ and $ux$ and equals $u(Q_i)$
  when specialized to $(t,u,x) = (t(P_i),u(P_i),t(Q_i))$.
\end{enumerate}
\begin{remark} 
Typically we use {\sf AnalyticJacobian} and {\sf EndomorphismRing} to carry out step (1), and {\sf ToAnalyticJacobian} and {\sf FromAnalyticJacobian} to carry out step (3).  The remainder of the algorithm requires only the matrix $T^{an}_\Omega$ (and not $T^{an}_\Z$) which could also be obtained using the algorithm in \cite{KM} rather than {\sf EndomorphismRing}. 
\end{remark}
\begin{remark}
We do not carry out a detailed analysis of the floating point
precision needed or the running time of our algorithm.  We remark that
$400$ digits of precision were sufficient for the examples in this
paper and that the machine used to perform the computations in this
paper (4 GHz, 32 GB RAM) completed the entire sampling and
interpolation process for individual correspondences in minutes.  For
our most complicated example, presented in Theorem \ref{thm4}, CPU
time was under two minutes.

To be able to carry out these steps, we need a large supply of sample
points, and sufficient precision.  As far as the number of sample
points needed for interpolation to find the equation of $Z$, we
closely follow the argument of \cite[Section 3]{vW1}. There it is
observed that the coefficients of $F$ (which are $x_1+x_2$ and $x_1
x_2$ in the notation of \cite{vW1}) are rational functions in $t$ and
have degrees which are bounded by the intersection number of
$\alpha(\Theta)$ and $2 \Theta$.  In our case, this equals
$\tr_{\Q(\sqrt{D})/\Q}(\alpha^2)$. Consequently, we choose $\alpha \in
\sO_D$ for which the trace is minimized: $\alpha = \pm \sqrt{D}/2$ if
$D$ is even, and $\alpha = (\pm 1 \pm \sqrt{D})/2$ if $D$ is odd. In
practice, since the degrees of the functions involved may be quite a
bit smaller than the upper bound, it is more efficient both in terms
of time and computer memory to choose a small sample size and attempt
to see if the computation of $Z$ succeeds.
\end{remark}

\begin{remark}
  From the equations for $Z$ and the maps $\phi$, $\psi$ to $C$, we can
  compute the action of $T$ on divisors of the form $P-Q$ using
  Equation \ref{eqn:Tondiv}.  We can then use standard equations
  \cite{CF} for the group law on $\Jac(C)$ to extend this formula to
  arbitrary divisor classes of degree zero.  Similarly, we can compute
  the algebraic action of an arbitrary element $m+nT \in \Z[T]$ of the
  real quadratic order using formulas for the group law.
\end{remark}

\paragraph{Example.}  
We conclude this section with an example for discriminant $5$.  Let $K
= \Q$ and let $C$ be the genus two curve defined by
\begin{equation}
\label{eqn:Cex1}
C : u^2 = h(t) \mbox{ where } h(t) = t^5 - t^4 + t^3 + t^2 - 2t + 1.
\end{equation}
The Jacobian of $C$ corresponds to the point $(g,h) = \left(
-\tfrac{8}{3},\tfrac{47}{2} \right)$ in the model computed in
\cite{EK} for the Hilbert modular surface $Y(\sO_5)$ parametrizing
principally polarized abelian surfaces with an action of $\sO_5$.  By
the method outlined above, we discover that the degree two branched
cover $f: Z \to C$ defined by
\begin{equation}
\label{eqn:Fex1}
F(x) = 0 \mbox{ where } F(x) = t^2 x^2-x-t+1 \in K(C)[x]
\end{equation}
is a correspondence associated to real multiplication by $\sO_5$.  In fact, setting
\begin{equation}
\label{eqn:yex1}
y = \frac{1}{t^3} u-\frac{t+1}{t^3} u x \in K(Z)
\end{equation}
we find that $y$ is a square root of $h(x)$ in $K(Z)$, and the map
$\psi(t,u,x) = (x,y)$ defines a second map $Z \to C$.  In Section
\ref{sec:oneforms} we will prove that $T = \psi_* \circ \phi^*$ generates
real multiplication by $\sO_5$, thereby certifying these equations for
$Z$. We note here that the degree of $\psi$ is $2$ since, fixing $x$,
there are two choices for $t$ satisfying Equation \ref{eqn:Fex1}, and
$u$ is determined by $(x,y,t)$ by Equation \ref{eqn:yex1}.  The curve
$Z$ has genus $6$, as can be readily computed in {\sf Magma} or {\sf
  Maple}.
\begin{remark}
It would be interesting to use the tools in this paper to study the
{\em geometry} of correspondences over Hilbert modular surfaces.  In
particular, one might explore how the geometry of $Z$ varies with $C$
and $T$, and how $Z$ specializes at curves $C$ lying on arithmetically
and dynamically interesting loci such as Teichm\"uller curves and
Shimura curves.  For instance, compare Theorems \ref{thm:ex1} and
\ref{thm5} for discriminant $5$ and Theorems \ref{thm2} and \ref{thm3}
for discriminant $12$.
\end{remark}

\section{Minimal polynomials and action on one-forms}
\label{sec:oneforms}
In this section, we describe how to certify the equations we
discovered by the method in Section \ref{sec:discovery}.  We have now
determined an equation for a curve $Z$ with an obvious degree two map
$\phi: Z \to C$ given by $\phi(t,u,x) = (t,u)$.  We have also computed
equations for a second map $\psi: Z \to C$ given by $\psi(t,u,x) = (x,y)$.

We will now describe how to compute the action $T_\Omega$ of $T =
\psi_* \circ \phi^*$ on $\Omega(C)$.  Since the representation of the
endomorphism ring of $\Jac(C)$ on $\Omega(C)$ is faithful, the minimal
polynomial for $T$ is equal to the minimal polynomial for $T_\Omega$.
Fixing $\omega \in \Omega(C)$, to compute $T_\Omega(\omega)$, we first
pullback along $\psi$ and then pushforward along $\phi$.  The order of
composition is reversed since the functor sending $C$ to the vector
space $\Omega(C)$ is contravariant, whereas the functor sending $C$ to
$\Jac(C)$ is covariant.  Pullbacks are straightforward, and the
pushforward along $\phi$ can be computed from the rule
\begin{equation}
\phi_*( v \,\eta ) = \tr(v) \omega \mbox{ when $\eta = \phi^*(\omega)$ and $v \in K(Z)$.}
\end{equation}
The trace on the right hand side is with respect to the field
extension $K(Z)$ over the subfield isomorphic to $K(C)$ associated to
the map $\phi$.  We now see that
\begin{equation}
\label{eqn:Tomega}
T_\Omega (\omega) = \phi_* \circ \psi^* (\omega) = \tr(\psi^*(\omega)/\phi^*(\omega)) \cdot \omega.
\end{equation}
The trace on the right hand side of Equation \ref{eqn:Tomega} is over
the field extension $K(t,u,x)/K(t,u)$ and can be computed
easily from the equations defining $\phi$.  We return to the example
in the previous section.

\paragraph{Example.}  
Let $Z \subset C \times C$ be the correspondence defined by Equations
\ref{eqn:Cex1}, \ref{eqn:Fex1} and \ref{eqn:yex1}.  Let $\omega_1 =
dt/u$ and $\omega_2 = t\,dt/u$ be the standard basis for $\Omega(C)$.
To compute the action of $T_\Omega$ on $\Omega(C)$, we need to work
with the function field $K(Z)$ and its derivations.  The derivations
form a one dimensional vector space over $K(Z)$.  It is spanned by
both $dx$ and $dt$, and the relation between them is computed by
implicitly differentiating Equation \ref{eqn:Fex1}.  We compute
\begin{equation}
\label{eqn:gsoverfs}
\frac{\psi^*(\omega_1)}{\phi^*(\omega_1)} = \frac{dx/y}{dt/u} = \frac{(-2 t^4 + t^3 - t^2) x + (-2t^4 + 4t^3 - t^2 - t + 1)}{4t^3 - 4t^2 + 1}.
\end{equation}
We now need to compute the trace of the right hand side over $K(t,u)$.
From Equation \ref{eqn:Fex1}, we see that the trace of $x$ is $1/t^2$,
and therefore the trace of the right hand side of Equation
\ref{eqn:gsoverfs} is $(1-t)$.  We conclude that
\begin{equation}
T_\Omega(dt/u) = (1-t) dt/u.
\end{equation}
Similarly, we compute that $T_\Omega(t\,dt/u) = -dt/u$ and hence the
matrix for $T_\Omega$ is
\begin{equation}
M = \begin{pmatrix}
1 & -1 \\
-1 & 0
\end{pmatrix}.
\end{equation}
The minimal polynomial for $M$, hence for $T_\Omega$ and $T$ as well,
is $T^2 - T -1$.  We conclude that $T$ generates a ring $\Z[T]$
isomorphic to $\sO_5$.

\paragraph{Rosati involution.}  
The adjoint for $T$ with respect to the Rosati involution is the
endomorphism
\begin{equation} T^\dagger = (\psi_* \circ \phi^*)^\dagger = \phi_* \circ \psi^*. \end{equation}
By computing the action of $T^\dagger$ on $\Omega(C)$ by the procedure
above, we verify that $T^\dagger_\Omega = T_\Omega$ and conclude that
$T = T^\dagger$ is self-adjoint with respect to the Rosati involution.

We note that the proof of Theorem~\ref{thm:ex1} stated in the
introduction is now complete.

\section{Further examples}
\label{sec:examples}
We now describe several other examples of the results one can obtain via
our method. We choose relatively simple curves and small discriminants
for purposes of illustration. For instance, the first two examples
have Weierstrass points at $\infty$, and the others have two rational
points at $\infty$.  Each of the theorems stated in this section are
proved by carrying out an analysis similar to our analysis of the
curve defined by Equation \ref{eqn:Cex1} in Sections
\ref{sec:discovery} and \ref{sec:oneforms}.  We provide computer code
in the auxiliary files to carry out these analyses.

Our first example involves a quadratic ring of slightly larger
discriminant.
\begin{theorem} \label{thm2}
Let $C$ be the curve defined by $u^2 = t^5-6t^4 + 15 t^3-22 t^2 + 17
t$ and let $\phi: Z \to C$ be the degree two branched cover defined by
\[
t(t^2-3t+1)^2x^2 - (4t^5-23t^4+46t^3-37t^2+6t+17) x + 4t(t^4-6t^3+15t^2-22t+17) = 0.
\]
The curve $Z$ is of genus $12$ and admits a map $\psi : Z \to C$ of
degree $5$.  The induced endomorphism $T = \psi_* \circ \phi^*$ of $\Jac(C)$
satisfies $T^2-3=0$ and generates real multiplication by
$\sO_{12}$.
\end{theorem}
\begin{figure}
   \includegraphics[scale=0.4]{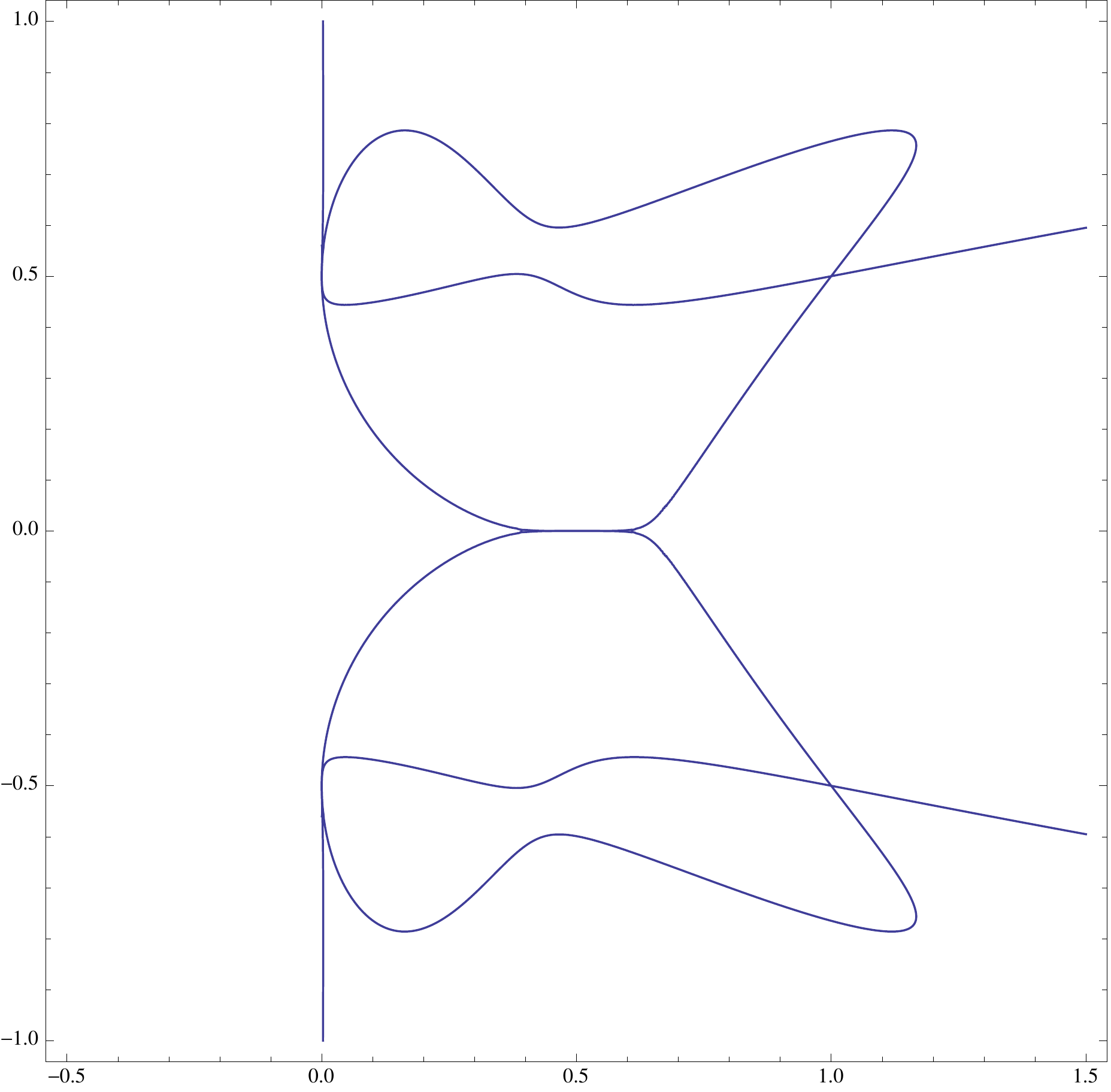} \sfcaption{\label{fig:Zex2}
     We plot the curve $Z$ of Theorem~\ref{thm2} in the
     $(1/x,1/u)$-plane.}
\end{figure}

\begin{remark}
The curve $C$ in Theorem \ref{thm2} corresponds to the point $(e,f) =
\left(\tfrac{34}{27},-\tfrac{5}{3}\right)$ on $Y(\sO_{12})$ in the
coordinates of \cite{EK}. The proof of the theorem proceeds along similar
lines as that of Theorem~\ref{thm:ex1}. The map $\psi: Z \to C$ takes
$(t,u,x)$ to $(x,y)$ where
\[
y =  -\frac{t^6-5t^5+12t^4-21t^3+32t^2-17t-17}{t^2(t^2-3t+1)^3} xu  + \frac{2(t-2)(t^4-2t^3-t^2-2t+17)}{t(t^2-3t+1)^3} u.
\]
The function field of $Z$  is generated by $x$ and $u$. We depict a plane model for $Z$ in Figure \ref{fig:Zex2}.  
\end{remark}

Our next example illustrates how the method developed in this paper
may be used to identify eigenforms and determine points on
Teichm\"uller curves.  Recall that $W_D$ is the moduli space of genus
two eigenforms for $\sO_D$ with a {\em double zero}, and is a disjoint
union of Teichm\"uller curves \cite{Mc1}.
\begin{theorem} \label{thm3}
Let $C$ be the curve $u^2 = t^5 - 2 t^4 - 12 t^3 - 8 t^2 + 52 t + 24$
and let $\phi: Z \to C$ be the degree two branched cover defined by
\begin{align*}
16(t-2)(t+1)^2x^2 -(3t^4+16t^3+12t^2-192t-164) x & \\
 + (9t^5-12t^4-140t^3-48t^2+276t+16) &= 0.
\end{align*}
The curve $Z$ is of genus $11$ and admits a map $\psi: Z \to C$ of degree
$5$. The induced endomorphism $T = \psi_* \circ \phi^*$ of $\Jac(C)$
satisfies $T^2-3=0$ and generates real multiplication by
$\sO_{12}$. Moreover, the moduli point corresponding to $C$ on
$Y(\sO_{12})$ lies on the Weierstrass-Teichm\"uller curve $W_{12}$.
\end{theorem}

\begin{remark}
The curve $C$ in Theorem \ref{thm3} was shown in \cite{EK} to have a
Jacobian which admits real multiplication by $\sO_{12}$.  In fact, it
corresponds to the point $(e,f) = \left( -\tfrac{3}{8},
-\tfrac{1}{2}\right)$ on $Y(\sO_{12})$ in the coordinates of
\cite{EK}. In \cite{KM}, we showed using the Eigenform Location
Algorithm that $dt/u$ is an eigenform for $\sO_{12}$ to conclude that
this point lies on $W_{12}$.
The equations for $Z$ and the maps $\phi, \psi$ yield an alternate
proof of both of these facts. For completeness, we note the expression
for the rational function $y$ on $Z$ needed to define $\psi: (t,u,x)
\mapsto (x,y)$ is
\[
y =  - \frac{11t^4-24t^3+12t^2-112t-132}{64(t-2)^2(t+1)^3} xu \sqrt{3}  - \frac{15t^5-28t^4-36t^3+288t^2-52t-144}{64(t-2)^2(t+1)^3} u  \sqrt{3}.
\]
The rest of the verification is carried out in the computer code.
\end{remark}
Our next example involves a genus two curve without a rational
Weierstrass point.  The resulting correspondence is more complicated,
but still well within the reach of our method.
\begin{theorem} \label{thm4}
Let $C$ be the curve $u^2 = t^6+t^5+7t^2-5t+4$,
and let $\phi : Z \to C$ be the degree two branched cover defined by
\begin{align} \label{eqn:sextic}
& (3t^3-t^2+t+1)(368t^4-597t^3-233t^2+233t+41) x^2  \nonumber \\
&+ x \big( 4(199t^4-31t^3-185t^2-33t+6)  u  \nonumber \\
&  \qquad +2(430t^7-1601t^6+876t^5-623t^4-338t^3+257t^2-168t-65) \big) \\
& + 4(138t^5-153t^4-21t^3+55t^2+3t-18)  u \nonumber \\
& +  552t^8-1616t^7-1435t^6+4654t^5-3949t^4+900t^3+1035t^2-690t+21 = 0. \nonumber
\end{align}
The curve $Z$ is of genus $11$ and admits a map $\psi$ to $C$ of degree
$4$. The induced endomorphism $T = \psi_* \circ \phi^*$ of $\Jac(C)$ satisfies $T^2-2=0$ and generates real multiplication by
$\sO_8$.
\end{theorem}

\begin{remark}
\label{rmk:sextic}
The curve $C$ in Theorem \ref{thm4} corresponds to the point $(r,s) =
\left( \tfrac{1}{8},\tfrac{59}{32} \right)$ on $Y(\sO_8)$ in the
coordinates of \cite{EK}.  Note that the coefficients of the
polynomial $F$ defining $Z$ are not in $K(t)$, in contrast to the case
where $C$ is a quintic hyperelliptic curve.  This is because the
hyperelliptic involution does not preserve the chosen point at infinity
$P_0$, and therefore does not commute with the deck transformation of
$\phi: Z \to C$.  Therefore, the discovery part of our algorithm in
which we compute equations for $Z$ has to be modified slightly.  The
coefficients of $F$ can be computed by determining a $K$-linear
relation between $1$, $t(Q_i)$, $t(Q_i)^2$, $u_i=u(P_i)$ and $u_i
t(Q_i)$ by LLL for each $i$ (rather than between the first three
quantities as in the quintic case).  The coefficients in these
relations are values of rational functions specialized at $t_i =
t(P_i)$, and we can interpolate to determine these rational functions
exactly.  A similar modification must be made to solve for $y \in
K(Z)$. For brevity, we have omitted the expression for $y$ here,
although it is available in the computer files.
\end{remark}

\section{Correspondences in families}
\label{sec:disc5family}
In this section we describe a correspondence on a universal family of
genus two curves over the entire Hilbert modular surface
$Y(\sO_5)$. There is one significant obstacle to implementing the
method described in Section \ref{sec:discovery} in families.  Suppose
$\left\{ C_\mu : \mu \in U \right\}$ is a family of curves
parametrized by the base $U$ each of which admits real multiplication
by $\sO$.  The method described in Section \ref{sec:discovery} allows
us to compute a correspondence $Z_\mu$ over $C_\mu$ for any particular
$\mu \in U$.  However, the first step in computing $Z_\mu$ involves a
choice of analytic Jacobian endomorphism $T^{an}_\mu$ generating
$\sO$.  There are typically two choices for $T^{an}_\mu$ with a given
minimal polynomial, and it is important to make these choices so that
the matrices $T^{an}_{\mu,\Omega}$ vary continuously in $\mu$ and the
$Z_\mu$'s are fibers of a single family.

To overcome this obstacle we first normalize the entire
family so that $dt/u$ and $t\,dt/u$ are eigenforms, using the
Eigenform Location Algorithm in \cite{KM}.  Then we simply choose
$T^{an}_\mu$ to have $T^{an}_{\mu,\Omega}$ equal to a constant
diagonal matrix.  Having consistently chosen $T^{an}_\mu$ in this way,
we compute $Z_\mu$ for various values of $\mu$ and interpolate to
determine a correspondence over the entire family.  The result is the
following theorem.
\begin{theorem} \label{thm5}
For $(p,q) \in \C^2$, let $C(p,q)$ be the curve defined by the
equation
\begin{equation}
\label{eqn:Cpq}
 \qquad u^2 = t^6 + 2p t^5 + 10q t^3 + 10q^2 t - 5(p-1)q^2
\end{equation} 
and let $\phi: Z(p,q) \to C(p,q)$ be degree two branched cover defined by
\begin{align} \label{eqn:sexticfamily}
& (2t-p)\big(4t +(3+\alpha)p\big)  x^2 +\big( 2(-\alpha-1)u \nonumber \\
& +\alpha(2t^3-2pt^2+p^2t+2q)  -(6t^3-6pt^2-p^2t-10q) \big)x \\
& - 2\big((1-\alpha)t -p\big)u + \alpha(2t^4-p^2t^2+6qt-4pq)  -( 2t^4-2pt^3+3p^2t^2-10qt+10pq ) = 0 \nonumber
\end{align}
where $\alpha = \sqrt{5}$. For generic $(p,q) \in \C^2$, the curve
$Z(p,q)$ is of genus $8$ and admits a holomorphic map $\psi$ to
$C(p,q)$ of degree $3$. The endomorphism $T = \psi_* \circ \phi^*$ of
$\Jac(C(p,q))$ is self-adjoint with respect to the Rosati involution
and generates real multiplication by $\sO_5$.
\end{theorem}
\noindent The only complication in the proof of Theorem \ref{thm5} is
that we need to work in the function field over the base field
$\Q(p,q)$ rather than $\Q$.  We provide computer code in the auxiliary
files to carry out the certification as in our previous examples. For
brevity, we have omitted the lengthy expression for $y$ in the map
$\psi: (t,u,x) \mapsto (x,y)$; it is available from the computer
files. For that choice of $y$, the endomorphism $T$ has minimal
polynomial $T^2-T-1$. (Replacing $y$ with $-y$ gives rise to an
endomorphism with minimal polynomial $T^2+T-1$.)

\begin{remark}
The coordinates for $Y(\sO_5)$ in Theorem \ref{thm5} are related to
the coordinates $(m,n)$ appearing in \cite{EK} by
\begin{equation}
(p,q) = \big(m^2/5 - n^2,(m-\alpha n)(5n^2-m^2)(5n^2-m^2+5)/125\big).
\end{equation}
In particular, our coordinates are quadratic twists of those appearing
in \cite{EK}.  This is because they are adapted to the eigenform
moduli problem, not the real multiplication moduli problem.  The field
of definition of a point $(p,q)$ is the field of definition of the
eigenforms $dt/u$ and $t\,dt/u$, which need not agree with the field
of definition of real multiplication.  In fact, these moduli problems
are isomorphic over $\Q(\sqrt{5})$, but not over $\Q$.  This also
explains the appearance of $\alpha = \sqrt{5}$ in the equation
defining $Z(p,q)$.
\end{remark}

\section{Divisor classes supported at eigenform zeros} \label{sec:eigenformzeros}
We now turn to the applications in dynamics for our equations for real
multiplication stated in the introduction.  Recall that $L$ is the
multisection in the universal Jacobian over $\M_{2,1}(\sO_5)$ whose
values at the pointed curve $(C,P)$ are divisors of the form in
Equation \ref{eqn:multisection}.  Our goal is to prove that the locus
$\M_{2,1}(\sO_5;L)$ defined by the vanishing of $L$ is an irreducible
surface in $\M_{2,1}$ and that $\M_{2,1}(\sO_5;L)$ is the closure of a
complex geodesic for the Teichm\"uller metric.

\paragraph{Marking eigenform zeros.}
We start by passing to a cover of $\M_{2,1}(\sO_5)$ on which we can
describe the multisection $L$ in terms of sections.  To that end we
define $\M_2^{\mathrm{ze}}(\sO_5)$ to be the space of pairs $(C,Z)$
where $C \in \M_2(\sO_5)$ and $Z \in C$ is a zero of an eigenform for
real multiplication by $\sO_5$.  Similarly we define
$\M_{2,1}^{\mathrm{ze}}(\sO_5)$ to be the pointed version consisting
of triples $(C,P,Z)$ with $(C,P) \in \M_{2,1}(\sO_5)$ and $(C,Z) \in
\M_2^{\mathrm{ze}}(\sO_5)$.  Here we are allowing $Z = P$.

The space $\M_{2}^{\mathrm{ze}}(\sO_5)$ is birational to the Hilbert
modular surface $Y(\sO_5)$.  To see this, fix $\gamma \in \sO_5$
satisfying $\gamma^2-\gamma-1 = 0$.  A point $(C,Z) \in
\M_2^{\mathrm{ze}}(\sO_5)$ determines a Rosati invariant 
endomorphism $T_\gamma(C,Z)$ of $\Jac(C)$ by
the requirement that the line of one-forms on $C$ vanishing at $Z$ are
$\gamma$-eigenforms for $T_\gamma(C,Z)$.  The map $(C,Z) \mapsto
(\Jac(C),T_\gamma(C,Z))$ is birational.  In particular,
$\M_2^{\mathrm{ze}}(\sO_5)$ is an irreducible surface.

\paragraph{Sections.}  Let $\eta$ be the hyperelliptic involution on $C$.  We can now define a section $L_\gamma$ of the universal Jacobian over $\M_{2,1}^{\mathrm{ze}}(\sO_5)$ by the formula
\begin{equation}
\label{eqn:section}
L_\gamma(C,P,Z) = (P-Z) - T_\gamma(C,Z) \cdot (\eta(Z) - Z) \in \Jac(C).
\end{equation}
Let $\M_{2,1}^{\mathrm{ze}}(\sO_5;L_\gamma)$ denote the locus in
$\M_{2,1}^{\mathrm{ze}}(\sO_5)$ where $L_\gamma$ vanishes.  Similarly,
we define $T_{1-\gamma}(C,Z)$, $L_{1-\gamma}$ and
$\M_{2,1}^{\mathrm{ze}}(\sO_5;L_{1-\gamma})$ by replacing $\gamma$
with its Galois conjugate $1-\gamma$.  From the definition of the
multisection $L$, it is clear the map forgetting $Z$ sends the union
of $\M_{2,1}^{\mathrm{ze}}(\sO_5;L_\gamma)$ and
$\M_{2,1}^{\mathrm{ze}}(\sO_5;L_{1-\gamma})$ onto $\M_{2,1}(\sO_5;L)$.
In fact, each of these spaces individually maps onto
$\M_{2,1}(\sO_5;L)$ since the sections $L_\gamma$ and $L_{1-\gamma}$
are related by $L_\gamma(C,P,Z) = L_{1-\gamma}(C,P,\eta(Z))$.  We
record this fact in the following proposition.
\begin{proposition}
The space $\M_{2,1}^{\mathrm{ze}}(\sO_5;L_\gamma)$ maps onto $\M_{2,1}(\sO_5;L)$.
\end{proposition}

We will now use our equations for real multiplication to show that
$\M_{2,1}^{\mathrm{ze}}(\sO_5;L_\gamma)$ is a section of
$\M_{2,1}^{\mathrm{ze}}(\sO_5) \to \M_2^{\mathrm{ze}}(\sO_5)$.
\begin{proposition}
\label{prop:section}
For each $(C,Z) \in \M_2^{\mathrm{ze}}(\sO_5)$, there is a unique
solution $P \in C$ to the equation $L_\gamma(C,P,Z) = 0$.
\end{proposition}
\begin{proof}
The uniqueness is easy and does not require our equations for real
multiplication.  If $P_1,P_2$ are solutions to $L_\gamma(C,P,Z) = 0$,
then $P_1-P_2$ is a principal divisor.  Since the smooth genus two
curve $C$ admits no degree one rational map, we must have $P_1=P_2$.

The locus in $\M_2^{\mathrm{ze}}(\sO_5)$ consisting of pairs $(C,Z)$
which admit a solution to $L_\gamma(C,P,Z) = 0$ is closed.  This
follows from the fact that $\M_{2,1}^{\mathrm{ze}}(\sO_5) \to
\M_2^{\mathrm{ze}}(\sO_5)$ is a projective map and, since
$\M_2^{\mathrm{ze}}(\sO_5)$ is irreducible, it is enough to check that
the generic pair $(C,Z) \in \M_2^{\mathrm{ze}}(\sO_5)$ admits such a
solution.

Recall the notation of Theorem \ref{thm5} and its proof in the
auxiliary files.  For generic $(p,q) \in \C^2$, we have a genus two
curve $C(p,q)$, a Rosati invariant endomorphism $T(p,q)$ of $\Jac(C(p,q))$ satisfying $T(p,q)^2-T(p,q)-1
= 0$, and a $T(p,q)$-eigenform $\omega(p,q) = t\,dt/u$ with eigenvalue
$\gamma = (1+\alpha)/2$.  To mark a zero of $\omega(p,q)$, we choose
$z$ a square root of $5(1-p)$ and set
\begin{equation}
Z(z,q) = (0,zq) \in C(p,q).
\end{equation}
Counting dimensions, we see that the $(z,q)$-plane parametrizes an
open subset of $\M_2^{\mathrm{ze}}(\sO_5)$ by the formula $(z,q)
\mapsto (C(z,q),Z(z,q))$.  We further define
\begin{equation}
\label{eqn:Pzq}
P(z,q) = \left( 2(1-p),z(8-16p+8p^2+5q)/\alpha \right)  \in C(z,q).
\end{equation}
Using our equations for the correspondence defining $T(z,q)$ and
Equation \ref{eqn:Tondiv}, we compute the divisor $T(z,q) \cdot
(\eta(Z(z,q)) - Z(z,q))$.  Combined with standard formulas for the
group law on $\Jac(C)$ (which have been implemented in {\sf Magma}),
we verify that $L_\gamma(C(z,q),P(z,q),Z(z,q)) = 0$.  We include code
in the auxiliary files to verify this equation.
\end{proof}
We are now ready to prove Theorem \ref{thm:eformzeros}.
\begin{proof}[of Theorem \ref{thm:eformzeros}]
The locus $\M_{2,1}^{\mathrm{ze}}(\sO_5;L_\gamma)$ is biregular to the
irreducible surface $\M_2^{\mathrm{ze}}(\sO_5)$ by Proposition
\ref{prop:section}, and maps onto $\M_{2,1}(\sO_5;L)$ by a map of
finite degree.  Therefore $\M_{2,1}(\sO_5;L)$ is an irreducible
surface in $\M_{2,1}$.
\end{proof}

\paragraph{Complex geodesics in moduli space.}
We now prove Theorem \ref{thm:cxgeod} about geodesics in $\M_{2,1}$
which is a corollary of Theorem \ref{thm:eformzeros}.  We refer the
reader to the survey articles \cite{Wr} and \cite{Z} for background on
geodesics in the moduli space of curves.

\begin{proof}[of Theorem \ref{thm:cxgeod}]
Fix a curve $C \in \M_2(\sO_5)$ and an $\sO_5$-eigenform $\omega$ on
$C$.  The form $\omega$ generates a complex geodesic $f_\omega : \hh
\to \M_2$ with $f_\omega(i) = C$ and $f_\omega'(i)$ tangent to
$\M_2(\sO_5)$.  By \cite{Mc1}, the image of $f_\omega$ is contained in
$\M_2(\sO_5)$.  We choose $C$ and $\omega$ generically so that
$\overline{f_\omega(\hh)} = \M_2(\sO_5)$ (cf. \cite{Mc2}).

The values of $f_\omega$ are related to $C$ by Teichm\"uller mappings.
In particular, there is a distinguished holomorphic one-form
$\omega_\tau$ (up to scale) on $C_\tau = f_\omega(\tau)$ and a
homeomorphism $C \to C_\tau$ which is affine for the singular flat
metrics $|\omega|$ and $|\omega_\tau|$.  The zeros of $\omega$ are in
bijection with those of $\omega_\tau$ via the Teichm\"uller mapping
and, by \cite{Mc1}, $\omega_\tau$ is also an $\sO_5$-eigenform.  We
conclude that there is a holomorphic zero marked lift
\begin{equation}
f_\omega^{\mathrm{ze}} : \hh \to \M_2^{\mathrm{ze}}(\sO_5)
\end{equation}
whose composition with the map forgetting $Z$ equals $f_\omega$.
Composing $f_\omega^{\mathrm{ze}}$ with the section
$\M_2^{\mathrm{ze}}(\sO_5) \to \M_{2,1}^{\mathrm{ze}}(\sO_5;L_\gamma)$
and the map forgetting $Z$, we obtain a map
\[ f_\omega^P : \hh \to \M_{2,1} \]
which is a section of $f_\omega$ over $\M_{2,1} \to \M_2$.

There are several ways to conclude that $f_\omega^P$ is a complex
geodesic.  The map $f_\omega^P$ is a section over the complex geodesic
$f_\omega$, and such sections are complex geodesics by a well-known
argument relying on the equality of the Kobayashi and Teichm\"uller
metrics on $\M_{g,n}$.  Alternatively, for $(C_\tau,P_\tau) =
f_\omega^P(\tau)$ we have an $\sO_5$-eigenform $\omega_\tau$ and a
zero $Z_\tau$ of $\omega_\tau$ satisfying
$L_\gamma(C_\tau,P_\tau,Z_\tau) = 0$.  We conclude that the relative
periods
\[ \int_{Z_\tau}^{P_\tau} \omega_\tau, \mbox{ and } \gamma \int_{Z_\tau}^{\eta(Z_\tau)} \omega_\tau \]
differ by an absolute period of $\omega_\tau$.  Consequently, the
Teichm\"uller mapping from $C \to C_\tau$ sends $P$ to $P_\tau$ and
$f_\omega^P$ is a complex geodesic.

Thus we have a complex geodesic $f_\omega^P$ in $\M_{2,1}$ such that
$\overline{f_\omega^P(\hh)}$ lies in $\M_{2,1}(\sO_5;L)$ and maps onto
$\M_2(\sO_5)$.  Since both $\M_{2,1}(\sO_5;L)$ and $\M_2(\sO_5)$ are
irreducible surfaces, we must have $\overline{f_\omega^P(\hh)} =
\M_{2,1}(\sO_5;L)$.
\end{proof}
\begin{figure}
  \begin{center}
   \includegraphics[scale=0.35]{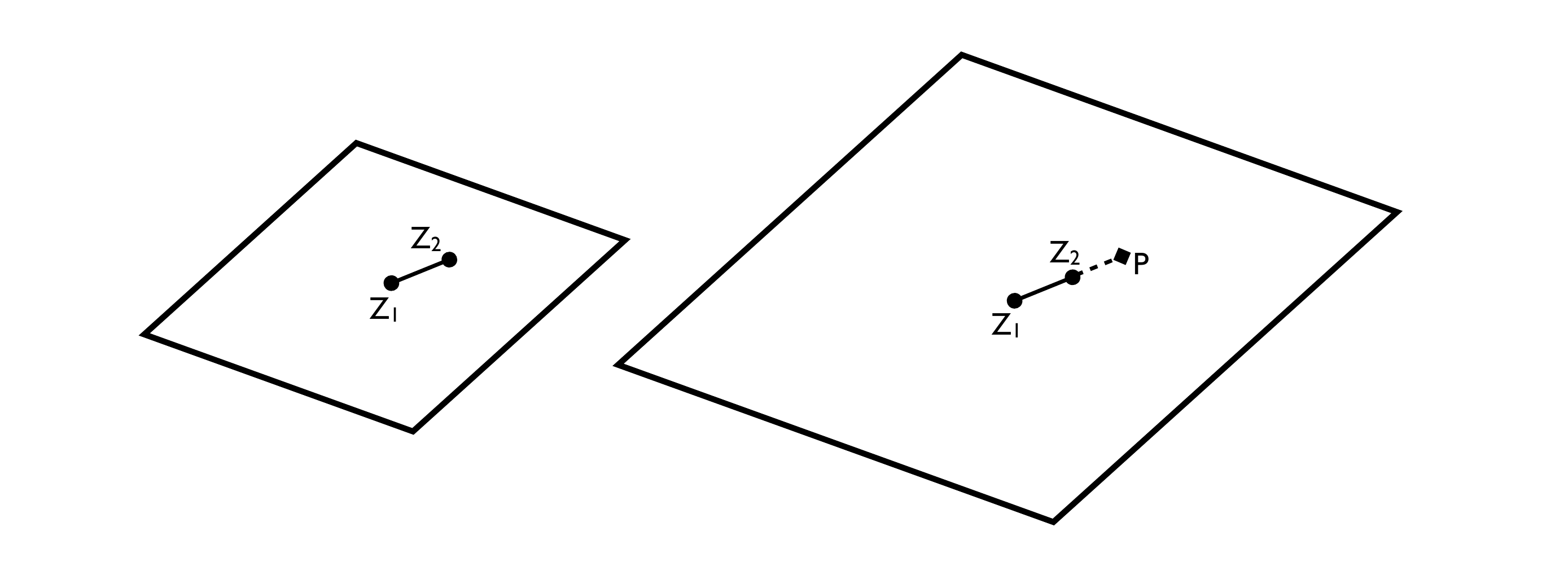}
\end{center}
\sfcaption{\label{fig:polygons} Genus two eigenforms for $\sO_5$ can
  be built out of a parallelogram $U \subset \C$ and the similar
  parallelogram $\gamma U$ by a connected sum.  The resulting form has
  zeros at $Z_1=0$ and $Z_2=t$ and the marked point $P=\gamma t$
  satisfies $L_\gamma(C,P,Z_1) = 0$.}
\end{figure}
\paragraph{Polygons and marked points.}
McMullen described how to polygonally present eigenforms for $\sO$ in
genus two \cite{Mc2}.  Set $\gamma = (1+\sqrt{5})/2$ to be the golden
mean. Eigenforms for discriminant $5$ are obtained from a
parallelogram $U \subset \C$ centered at $0$ and the similar
parallelogram $\gamma U \subset \C$ by gluing opposing sides on each
parallelogram and performing a connected sum along a straight line
interval $I$ connecting $0$ and $t \in U$.  The form $dz$ is invariant
under these gluing maps and the resulting quotient $(C,\omega) = (U
\#_I \gamma U,dz) / \sim$ is an $\sO_5$-eigenform.  Wright's
conjecture of the existence of a dynamically natural way to mark
curves in $\M_2(\sO_5)$ posited in particular that one could mark the
eigenform $(U \#_I \gamma U)/\sim$ at the point $P = \gamma t$ in the
polygon $\gamma U$ (see Figure \ref{fig:polygons}).

One way to see that the algebraically presented locus
$\M_{2,1}(\sO_5;L)$ in Theorem \ref{thm:cxgeod} equals the locus
polygonally presented by Wright is by first checking that they agree
somewhere, e.g. at the regular decagon eigenform which is the limit of
$(C(z,q),P(z,q))$ as $q \to 0$ in our parametrization.  The period
relations imposed by the vanishing of $L$ then imply that the points
marked in $\M_{2,1}(\sO_5;L)$ coincide with Wright's polygonal
description at a nearby generic point.  Therefore, the algebraic and
polygonal descriptions agree along an entire complex geodesic which is
dense in $\M_{2,1}(\sO_5;L)$.

\affiliationone{Abhinav Kumar \\
Department of Mathematics\\
Stony Brook University\\
Stony Brook, NY 11794 \\
USA
\email{thenav@gmail.com}}
\affiliationtwo{Ronen E. Mukamel \\
Department of Mathematics \\
Rice University, MS 136 \\
6100 Main St. \\
Houston, TX 77005 \\
USA
\email{ronen@rice.edu}}

\end{document}